\documentclass{amsart}

\usepackage[utf8]{inputenc}

\usepackage{amsmath}
\usepackage{amssymb}
\usepackage{mathtools}
\usepackage{hyperref}
\usepackage{xcolor}
\usepackage{tikz}

\newcommand{\bea}{\begin{eqnarray*}}
\newcommand{\eea}{\end{eqnarray*}}

\newcommand{\bm}{\begin{pmatrix}}
\newcommand{\fm}{\end{pmatrix}}
\newcommand{\bvm}{\begin{vmatrix}}
\newcommand{\fvm}{\end{vmatrix}}
\newcommand{\bbm}{\begin{bmatrix}}
\newcommand{\fbm}{\end{bmatrix}}

\newcommand\Z{\mathbb Z}

\newcommand\R{\mathbb R}

\newcommand\mc{\mathcal}

\newtheorem{theorem}{Theorem}
\newtheorem{ex}{Example}
\newtheorem{fact}{Fact}
\newtheorem{lemma}[theorem]{Lemma}
\newtheorem{cor}[theorem]{Corollary}

\newtheorem{defn}{Definition}

\newcommand{\set}[1]{\{#1\}}

\DeclareMathOperator{\Ex}{E}
\DeclareMathOperator{\Prob}{P}

\title{A matrix model for random nilpotent groups}
\author{Kelly Delp, Tullia Dymarz, Anschel Schaffer-Cohen}

\begin{document}
\maketitle

\begin{abstract}
We study random torsion-free nilpotent groups generated by a pair of random words of length $\ell$ in the standard generating set of $U_n(\mathbb{Z})$.
Specifically, we give asymptotic results about the step properties of the group when the lengths of the generating words are functions of  $n$. We show that the threshold function for asymptotic abelianness is $\ell = c \sqrt{n}$, for which the probability approaches $e^{-2c^2}$, and also that the threshold function for having full-step, the same step as $U_n(\Z)$, is between $c n^2$ and $c n^3$.

\end{abstract}


\section{Introduction}

The goal of this paper is to study random finitely-generated torsion-free nilpotent groups (also known as $T$-groups \cite{tgroups}). 
Recall that a nilpotent group $N$ is one for which the lower central series eventually terminates:
$$N = N_0 \geq N_1 \geq \dotsb \ge N_r =\set{0}$$
where $N_{i}=[N,N_{i-1}]$ is  the $i$th commutator subgroup (i.e. the subgroups generated by commutators of elements in $N$ and $N_{i-1}$). If $r$ is the first index with $N_r=\set{0}$ then  we say that $N$ is nilpotent of \emph{step} $r$. For more background on nilpotent groups see~\cite{mks}.

Our motivation for studying random nilpotent groups comes from Gromov's study of finitely generated random groups via random presentations (see \cite{Ollivier} for a detailed introduction). Roughly speaking Gromov considers groups $G_\ell$ given by a presentation $G_\ell =\left< S \mid R_\ell \right>$, where  the generating set $S$ is fixed and finite, and the relator set $R_\ell$  contains a subset of all possible relators of length at most $\ell$.  A random group is said to have a property $P$ if the probability that $G_\ell$ has $P$ goes to one as $\ell$ goes to infinity.  Generally the size of $R_\ell$ depends on $\ell$ and a chosen density constant $d \in [0,1]$ where $R_\ell$ at density $d$ contains on order of the $d$th power of possible relations of size less than $\ell$. Changing $d$ changes the properties of the random group. A fundamental result of Gromov's shows that when the density is greater than $1/2$  the resulting random group is trivial, and when the density is less than $1/2$ then the random group is a so-called \emph{hyperbolic} group. Unfortunately, nilpotent groups are not hyperbolic so this model is unsatisfactory for studying random nilpotent groups. For a recent generalization of Gromov's idea to quotients of free nilpotent groups see \cite{randnilp1}.

The model we study is motivated by a well-known theorem \cite{Hall} which states that any finitely-generated, torsion-free nilpotent group appears as a subgroup of $U_n(\Z)$, the group of $n \times n$ upper-triangular matrices with ones on the diagonal and entries in $\Z$.

 Let $E_{i,j}$ be the elementary matrix that differs from the identity matrix $I_n$ by containing a one at position $(i,j)$ and set $A_i=E_{i,i+1}$. Then the set $S=\set{A_1^{\pm1}, \dotsc, A_{n-1}^{\pm1}}$ of \emph{superdiagonal  elementary matrices} is the \emph{standard generating set} for $U_n(\Z)$. 
Our random subgroups will be generated by taking two simple random walks of length $\ell$ on the Cayley graph of $U_n(\Z)$ given by the generating set $S$.  These two random walks define two words, $V,W$ that generate a subgroup 
$$G_{\ell, n}:=\langle V, W\rangle \leq U_n(\Z).$$  
We are interested in the asymptotic properties of $G_{\ell,n}$ as $\ell \to \infty$.  For example, when $n$ is fixed one can show that the probability that $G_{\ell, n}$ is abelian goes to zero as $\ell \to \infty$.  If $\ell$ is a function of $n$, then the asymptotic abelianness depends on the rate of growth.

Before giving the precise statement of our results, we recall the Landau notation that we use to describe the growth rate of $\ell$:

\begin{itemize}
\item If $f(n) \in O(g(n))$ then there exist numbers $c$ and $N$, so that $n > N$ implies $f(n) < cg(n)$.
\item If $f(n) \in o(g(n))$ then for all $c>0$, there exists an $N$, so that $n > N$ implies $f(n) < cg(n)$.
\item If $f(n) \in \omega(g(n))$ then for all $c>0$, there exists an $N$, so that $n > N$ implies $f(n) > cg(n)$.
\end{itemize}

Additionally, we write $f(n) \sim g(n)$ if $\lim_{n \to \infty} f(n)/g(n)=1$.

Let $P$ be  a property of a group. For a particular length function $\ell(n)$, we say $G_{\ell,n}$ is \textit{asymptotically almost surely}  (a.a.s) $P$ if the probability that $G_{\ell,n}$ has $P$ approaches $1$ as $n$ approaches infinity. In Section \ref{sec:comm} we prove the following theorem:

\begin{theorem}\label{theorem:abelian} Let $G_{\ell,n}$ be a subgroup of $U_n(\Z)$ generated by two random walks of length $\ell$ in the standard generating set $S$ and suppose $\ell$ is a function of $n$.
\begin{enumerate}
\item If $\ell \in o(\sqrt{n})$ then asymptotically almost surely $G_{\ell, n}$  is abelian. 
\item If $\ell = c\sqrt{n}$ then the probability that $G_{\ell,n}$ is abelian approaches $e^{-2 c^2}$ as $n \rightarrow \infty$.
\item If $\ell \in \omega(\sqrt{n})$, then asymptotically almost surely $G_{\ell, n}$ is not abelian.
\end{enumerate}
\end{theorem}

Another  property we focus on in this paper is the step of $G_{\ell, n}$. Note that $U_n(\Z)$ is a step $n-1$ nilpotent group. We say that $G_{\ell, n}$ has \emph{full step} if it is also of step $n-1$.  
We show that the threshold function for being full step lies between $n^2$ and $n^3$.

\begin{theorem}\label{theorem:full_step} Let $G_{\ell,n}$ be a subgroup of $U_n(\Z)$ generated by two random walks of length $\ell$ in the standard generating set $S$ and suppose $\ell$ is a function of $n$.
\begin{enumerate}
\item If $\ell \in o(n^2)$ then asymptotically almost surely $G_{\ell, n}$ does not have full step.  
\item If $\ell \in \omega(n^3)$ then asymptotically almost surely $G_{\ell, n}$ has full step.  
\end{enumerate}
\end{theorem}

Theorem \ref{theorem:full_step} is proven in Section \ref{sec:full}. These theorems are summarized by the following diagram.

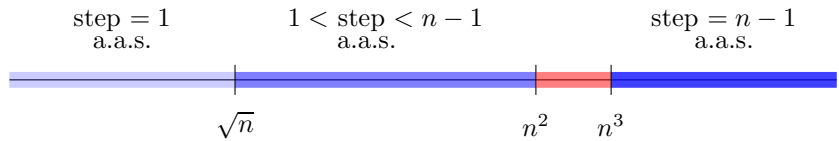
\begin{figure}[ht]
\begin{tikzpicture}
\draw  (0,0)--(11,0);
\draw [line width=6,blue,opacity= .2] (3,0)--(0,0);
\node at (1.5,.8) {step $=1$};
\node at (1.5,.5) {a.a.s.};
\draw [line width=6,blue,opacity=.5] (3,0)--(7,0);
\node at (5,.8) {$1<$ step $<n-1$};
\node at (4.75,.5) {a.a.s.};
\draw [line width=6,blue,style= nearly opaque] (8,0)--(11,0);
\node at (9.5,.8) {step $=n-1$};
\node at (9.5,.5) {a.a.s.};
\draw (3,.2)--(3,-.2);
\node at (3,-.6) {$\sqrt n$};
\draw (7,.2)--(7,-.2); 
\node at (7,-.6) {$n^2$};
\node at (8,-.6) {$n^3$};
\node at (7.5,.8) {};
\draw [line width=6,red,opacity=.5] (7,0)--(8,0);
\draw (8,.2)--(8,-.2);
\end{tikzpicture}
\end{figure}

\subsection{Outline} As random walks, $V,W$ are given by $V=V_1V_2 \cdots V_\ell$ and $W= W_1 W_2 \cdots W_\ell$ where $V_i, W_i \in S$. To prove Theorem \ref{theorem:abelian}, we define a sufficient condition for commuting, called \textit{supercommuting}. 

\begin{defn} Let $V = V_1 V_2 \dotsm V_{\ell}$ and 
$W = W_1 W_2 \dotsm W_{\ell}$ where $V_i$ and  $W_i$ are elements in the $U_n$ generating set $S$. The words $V$ and $W$ {\bf supercommute} if every $V_i$ commutes with every $W_j$.
\end{defn}

We show that when $\ell \in o(n)$, supercommuting and commuting are asymptotically equivalent, and that the threshold for supercommuting is at $\ell = c \sqrt{n}$.

For Theorem \ref{theorem:full_step}
most of the results are a matter of analyzing the entries on the superdiagonals of our generators $V$ and $W$. The $(i,i+1)$ superdiagonal entry of $V$, which we denote by $v_i$, is the sum over the number of $A_i^{\pm1}$ that occur in the walk,  where $A_i$ contributes $+1$, and its inverse $-1$. Therefore the vector of superdiagonal entries is the endpoint of a random walk in $\mathbb{Z}^{n-1}$; while these are well studied objects, most of the study has been on walks in a fixed dimension $n$. In our case, both the dimension $n$, and the length of the walk are going to $\infty$. We gather these results in Section \ref{sec:dist}.

{\bf Acknowledgements.} This project began in the Random Groups Research Cluster held at Tufts University in the summer of 2014, supported by Moon Duchin's NSF CAREER award DMS-1255442. We would like to thank the participants of this cluster for their questions, conversations, and attention in the summer of 2014, with special thanks to Meng-Che ``Turbo'' Ho, Samuel Lelièvre, and Mike Shapiro for their time in more extensive conversations. We would also like to thank Benedek Valk\'o for suggestions on Section 3.  
The first author also acknowledges support from National Science Foundation grant
DMS-1207296.

\section{Preliminaries}\label{sec:prelim}

Many of the results in this paper depend on the superdiagonal entries $v_{i,i+1}$ and $w_{i,i+1}$ of $V$ and $W$. For this reason we adopt the shorthand $z_i:= z_{i, i+1}$ for any matrix $Z$.

The following proposition gives a necessary condition for commuting in $U_n$.

\begin{lemma}\label{lemma:commute}
Let $W=[w_{i,j}]$ and $V=[v_{i,j}]$ be matrices in $U_n$. If  $W$ and $V$ commute then $w_i v_{i+1}=w_{i+1} v_i$ for all $1 \le i \le n-2$.
\end{lemma}

\begin{proof}
This is a straightforward computation. The first superdiagonal of $C=VWV^{-1}W^{-1}$ vanishes and the second superdiagonal entries are given by 
$c_{i,i+2}= w_{i+1} v_i - w_i v_{i+1}$.
\end{proof}

\begin{cor}\label{cor:elementary_commute}
The elementary superdiagonal matrices $A_i^{\pm1}$, $A_j^{\pm1}$ commute if and only if  $|i-j| \neq 1$.
\end{cor}

Next we study the  $k$th commutator subgroup of $G_{\ell,n}=\langle V,W\rangle$.
Note that in a nilpotent group the $k$th commutator subgroup is generated by 
all $m$-fold commutators for $m \geq k$ of the form 
$$[B_1 [B_2 \cdots [B_{m}, B_{m+1}]]] \cdot ]$$
where the $B_i$ are chosen from a fixed generating set (see for example Lemma 1.7 in \cite{Hall}). 
Therefore to test that $G_{\ell,n}$ is $k$-step nilpotent we only need to check 
 that $[B_1 [B_2 \cdots [B_{k}, B_{k+1}]]] \cdot ]=I$ when $B_i \in \{ V,W\}$. 
 

In Lemma \ref{lemma:commute} we noted that taking a commutator resulted in a matrix with zeros along the first superdiagonal. In the next lemma we show that taking a $k^{th}$ commutator results in zeros on the first $k$ superdiagonals. We also give a recursive formula for the entries on the $(k+1)^{\textrm{st}}$ superdiagonal using iterated two dimensional determinants.

\begin{lemma}\label{lemma:nested_comm} Let $C^k=[c^k_{i,j}]$ be a $k$-fold commutator of two matrices $V,W$; then $c^k_{i,j}=0$ when $i <  j \leq i+k$
and
\begin{equation}\label{eqn:nested}c^k_{i,k+i+1} = \det{\begin{bmatrix} z_{i,i+1} & c^{k-1}_{i,k+i}  \\   z_{k+i,k+i+1} & c^{k-1}_{i+1,k+i+1} \end{bmatrix} }\end{equation}
where $Z=[z_{i,j}]$ and either $Z=V$ or $Z=W$. 
\end{lemma}
\begin{proof} We prove this result by induction, where the base case is given in the proof of  Lemma \ref{lemma:commute}. Assume $C^{k-1}$ is given, and for convenience let $K = C^{k-1}$. Since the first $k-1$ superdiagonals of $K$ contain all zeros, computing $C^k = ZKZ^{-1}K^{-1}$ yields zeros on the first $k$ superdiagonals, and on the $(i,i+k+1)$-diagonal we have   $$z_{i,i+1} c^{k-1}_{i+1,k+i+1} -z_{k+i,k+i+1}c^{k-1}_{i,k+i}.$$
To help see this, note that when the first nonzero superdiagonals of $Z, C^{k-1},C^{k}$ are overlayed the resulting matrix is the following. 
$$\bm
\ddots &\vdots&\vdots&\vdots&\vdots&\vdots& \\
&1 & z_{i,i+1} & \cdots & c^{k-1}_{i,k+i} & c^k_{i,i+k+1} & \cdots \\
&& 1  & \cdots&& c^{k-1}_{i+1,k+i+1} & \cdots\\
&&& \ddots  &\vdots&\vdots&  \cdots \\
&&&  & 1& z_{k+i,k+i+1}&  \cdots \\
&&&&&1&  \cdots \\
&&&&&&  \ddots \\

\fm$$
\end{proof}

Our first application of Lemma \ref{lemma:nested_comm} is the following lemma, which shows that $G_{\ell, n}$ cannot be full step if $V,W$ have a matching pair of zeros on their superdiagonals. 
\begin{lemma}\label{lemma:not_full} Given $G_{\ell,n}=\langle V,W\rangle$,
if there is some $1 \leq d \leq n-1$ such that $v_d=w_d=0$ then the step of $G_{\ell,n}$ is bounded by $\max\{ d-1, n-1-d\}$.
\end{lemma}
\begin{proof}
Recall that $v_d=v_{d,d+1}$ and similarly  $w_d=w_{d,d+1}$.
By Lemma \ref{lemma:nested_comm} we have that  for $C^1=[V,W]$ 
$$c^1_{d-1,d+1}= \det{\begin{bmatrix} v_{d-1,d} & w_{d-1,d}  \\   v_{d,d+1} & w_{d,d+1} \end{bmatrix} }=0$$
since the bottom row of this two by two matrix has both entries to zero. Similarly
$$c^1_{d,d+2}= \det{\begin{bmatrix} v_{d,d+1} & w_{d,d+1}  \\   v_{d+1,d+2} & w_{d+1,d+2} \end{bmatrix} }=0$$
since the top row of the two by two matrix has entries both equal to zero.
Inductively, by Equation \ref{eqn:nested}, we have that $c^k_{d-k,d+1}=c^k_{d,d+k+1}=0.$ 
This is because either the top or bottom row of the matrix in Equation \ref{eqn:nested} will have both entries equal to zero. Alternatively, if both $c^{k-1}_{i,k+i}$ and $c_{i+1,k+i+1}^{k-1}$ are zero then $c^k_{i,k+i+1}=0$ since then the righthand column of the matrix in Equation \ref{eqn:nested} will have both entries zero. In particular, $c^k_{i,j}$ is zero if $k > \max\{d-1,n-1-d\}.$ 

\end{proof}

Lemma \ref{lemma:nested_comm} also leads us to define a modified determinant product which gives us a method to calculate the entries of the first nonzero superdiagonal of a iterated commutator product of upper triangular matrices given their first superdiagonal entries.
 
\begin{defn} Let $\vec{a}=(a_1, \ldots, a_{s})$ and $\vec{b}=(b_1, \ldots, b_m)$ be vectors with $s \geq m$ and set $s-m=p$; then $[\vec{a}\ \ \vec{b}]$ is the $m-1$ dimensional vector given by
$$[\vec{a}\ \ \vec{b}] := \left(  \bvm a_1 & b_1 \\ a_{p+2} & b_2 \fvm ,  \bvm a_2 & b_2 \\ a_{p+3} & b_3 \fvm, \cdots,  \bvm a_{m-1} & b_{m-1} \\ a_{p+m} & b_m \fvm  \right).$$
%
\end{defn}
\begin{lemma} Let $\vec{b}_i$ be the vector containing the $n-1$ main superdiagonal entries of an $n \times n$ unipotent matrix $B_i$ labeled from top left to bottom right. Then the $(k+1)^{\textrm{st}}$ superdiagonal entries of the $k$-fold commutator $[B_1 [B_2 \cdots [B_{k}, B_{k+1}]]]$ are given by the $(n-k)$ dimensional vector
$$[\vec{b}_1 [ \vec{b}_2  \cdots [\vec{b}_k \ \ \vec{b}_{k+1}] \cdots ].$$
\end{lemma}
This lemma can be proved by direct computation or by inspecting the proof of Lemma \ref{lemma:nested_comm}. To illustrate this result, consider the following examples, the second of which will be used in Section \ref{sec:full}.

\begin{ex} We consider the commutator $[D,[C,[A,B]]]$ where the superdiagonal entries of $A$ are given by $(a_1, \ldots a_{n-1})$ and similarly for $B,C,D$. The first three superdiagonals are all zero while
the fourth superdiagonal has entries given by 
$$\bm \bvm d_1 &
\bvm c_1 &
\bvm a_1 & b_1 \\
a_2 & b_2
\fvm\\
\\

c_3 &

 \bvm a_2 & b_2 \\
a_3 & b_3
\fvm
\fvm
\\
\\

d_4 & 
\bvm c_2 &
 \bvm a_2 & b_2 \\
a_3 & b_3
\fvm\\
\\

c_4 & 

 \bvm a_3 & b_3 \\
a_4 & b_4
\fvm
\fvm
\\
\fvm &,
\cdots , &
 \bvm d_{n-4} &
\bvm c_{n-4} &
\bvm a_{n-4} & b_{n-4} \\
a_{n-3} & b_{n-3}
\fvm\\
\\

c_{n-2} & 

 \bvm a_{n-3} & b_{n-3} \\
a_{n-2} & b_{n-2}
\fvm
\fvm
\\
\\

d_{n-1} & 
\bvm c_{n-3} &
 \bvm a_{n-3} & b_{n-3} \\
a_{n-2} & b_{n-2}
\fvm\\
\\

c_{n-1} & 

 \bvm a_{n-2} & b_{n-2} \\
a_{n-1} & b_{n-1}
\fvm
\fvm
\\
\fvm 
\fm.
$$ 
\end{ex}
\begin{ex}\label{ex:commutator} Consider the commutator
\[ \underbrace{[W, [W, \dotsc [W}_\text{$n-2$},V]]] \]
where $V,W$ are $n \times n$ upper triangular matrices with main superdiagonals given by the vectors $(v_1,\cdots, v_{n-1})$ and $(w_1, \cdots, w_{n-1})$ respectively.  Using the iterated determinant formula we see that the first nonzero superdiagonal has only one entry  and  is given by 
\[ K_1 v_1 w_2 w_3 \dotsm w_{n-1} + K_2 w_1 v_2 w_3 \dotsm w_{n-1} + \dotsb + K_{n-1} w_1 \dotsm w_{n-2} v_{n-1}\]
where each $K_i = \binom{n-1}{i}$ with alternating signs.
\end{ex}

\section{Distribution of the superdiagonal entries}\label{sec:dist}
In this section we examine the probability of finding zeroes on the superdiagonals of $V$ and $W$ when $\ell \in \omega(n)$. In order to emphasize the dependence on $\ell$ we write $V^\ell, W^\ell$ instead of $V,W$ and $v_k^\ell, w_k^\ell$ instead of $v_k,w_k$ for the superdiagonal entries. If we fix $n$ and $k$ we can model $v^\ell_k$ as the endpoint of a lazy random walk in $\Z$:
\[ v^\ell_k = \sum_{j = 1}^\ell x_j \]
where $x_j = \pm 1$ with probability $1/2n$ each and $x_j = 0$ with probability $(n-1)/n$. 
Likewise for any two $k_1\neq k_2$ we have an induced lazy random walk on $\Z^2$:
\[ \begin{pmatrix} v^\ell_{k_1} \\ v^\ell_{k_2} \end{pmatrix} = \sum_{j = 1}^\ell \begin{pmatrix} x_j \\ y_j \end{pmatrix} \]
where $(x_j, y_j)= (\pm 1, 0)$ or $(0, \pm 1)$ with probability $1/2n$ each, and $(x_j, y_j)= (0, 0)$ with probability $(n-2)/n$. 

Our goal is to estimate $\Prob(v^\ell_k=0)$ and $\Prob(v^\ell_{k_1}=v^\ell_{k_2}=0)$.  The proofs of the following lemmas follow the standard proofs of the local central limit theorem for lazy random walks on $\Z^d$   where special attention is paid to the dependence of the estimates on $n$. (See for example Section 2.3 in \cite{Lawler}). We reproduce them here because we were not able to find this exact formulation in the literature. Morally
we rewrite everything in terms of $\lambda=\ell/n$  and provide error estimates. We can do this as long as $\lambda \to \infty$---that is, when $\ell \in \omega(n)$.  

To make the results in this section more applicable later on, we define a constant $K = 1/\sqrt{2\pi}$.

\begin{lemma}\label{lemma:prob_zero} Suppose $\ell \in  \omega(n)$. Then for a fixed $1 \leq k \leq n$ we have 
\[ \Prob(v^\ell_k = 0) \sim K \sqrt{\frac{n}{\ell}} \]
\end{lemma}
\begin{proof}
We begin by noting that  
the characteristic function of $x_j$ is given by
$$\phi(t)=\Ex(e^{tix_j})=1-\frac{1}{n} + \frac{1}{2n}(e^{it}+e^{-it})=1-\frac{1}{n}(1-\cos{t})$$

and the characteristic function of $v_k^\ell$ which is
$$\phi(t)^\ell = \left(1- \frac{1}{n}(1- \cos{t})\right)^\ell.$$
Therefore 
$$P(v^\ell_k=0)= \frac{1}{2\pi} \int_{-\pi}^{\pi}  \left(1- \frac{1}{n}(1- \cos{t})\right)^\ell dt.$$

The methods used to estimate this integral are identical to the ones used in the more general proof of Lemma \ref{lemma:two_zero} below so we do not produce them here. The above integral is transformed to 
$$P(v_k^\ell = 0 )= \frac{\sqrt{\frac{n}{\ell}}}{\sqrt{2\pi}}\left( \int_\R e^{-s^2/2}ds + o(1)\right) $$
\end{proof}

Since $v^\ell_k$ and $w^\ell_k$ are independent we have the following corollary:

\begin{cor}\label{cor:prob_matching}
Suppose $\ell \in \omega(n)$. For fixed $k$, $\Prob(v_k = w_k = 0) \sim K^2 n/\ell$.
\end{cor}

Next we prove an estimate on the probability of having a pair of zeros in fixed coordinates $k_1 \ne k_2$. 

\begin{lemma}\label{lemma:two_zero} Suppose $\ell \in \omega(n)$. Then
for fixed $k_1 \ne k_2$, $$P(v_{k_1}^\ell= v_{k_2}^\ell=0) \sim K^2 \frac{n}{\ell}.$$
\end{lemma}
\begin{proof}


We begin by computing  
the characteristic function of $(x_j,y_j)$ which is given by
$$\phi(t_1,t_2)=\Ex(e^{i(t_1x_j+t_2y_j)})=1-\frac{1}{n}(1-\cos{t_1})-\frac{1}{n}(1-\cos{t_2})$$

and the characteristic function of $(v_{k_1}^\ell,v_{k_2}^\ell)=(\sum_{j=1}^\ell x_j, \sum_{j=1}^\ell y_j)$ which is
$$\phi(t)^\ell = \left(1-\frac{1}{n}(1-\cos{t_1})-\frac{1}{n}(1-\cos{t_2})\right)^\ell.$$
Therefore 
$$P(v^\ell_{k_1}=v_{k_2}^\ell=0)=\frac{1}{(2\pi)^2} \iint_{[-\pi,\pi]^2}  \left(1-\frac{1}{n}(1-\cos{t_1})-\frac{1}{n}(1-\cos{t_2})\right)^\ell dt_1dt_2.$$
$$=\frac{1}{(2\pi)^2} \iint_{[-\pi,\pi]^2}  \left(  1  -\frac{|\theta|^2}{2n} + \frac{1}{n} h(\theta)
   \right)^\ell d\theta$$
where $\theta=(t_1,t_2)$ and  $h(\theta)= \sum_{i=2}^\infty (-1)^i\frac{t_1^{2i}+t_2^{2i}}{(2i)!}\in O(|\theta|^4)$.

We use the Taylor expansion $\log(1+x)=\sum_{i=1}^\infty \frac{x^i}{i}$ that is valid for $|x|\leq 1$
to write
\begin{equation}\label{theeqn}\log(\phi(\theta))= \log\left( 1 \underbrace{-\frac{|\theta|^2}{2n} + \frac{1}{n} h(\theta)}_{x}\right)
= -\frac{|\theta|^2}{2n}+ 
{\frac{1}{n} h(\theta)+ f(\theta,{1}/{n})}\end{equation}
where
$$f(\theta,{1}/{n}) = \sum_{j=2}^\infty \frac{1}{j} \left(\frac{1}{n} \sum_{i=1}^\infty (-1)^i\frac{t_1^{2i}+t_2^{2i}}{(2i)!}\right)^j =O(|\theta|^4).$$
This expansion is valid for 
$$\left|-\frac{|\theta|^2}{2n} + \frac{1}{n} h(\theta)\right|=\frac{1}{n}\left|-\frac{|\theta|^2}{2} + h(\theta)\right|\leq 1$$ 
which holds as long as $|\theta|< \delta $ where $\delta$ does not depend on $n$. (It holds for $n=1$ and so it holds for all $n$). Let $\lambda=\ell/n$. 
Now use a change of variable $\theta= s/\sqrt{\lambda}= s \sqrt{\frac{n}{\ell}}$ in Equation \ref{theeqn} and multiply both sides by $\ell$ to get
$$\ell \log\left(\phi\left(s/ \sqrt{\lambda}\right)\right) = -\frac{|s|^2}{2}+ \underbrace{\frac{\ell}{n} h(s/ \sqrt{\lambda})+  \bar{f}(s,1/\ell,
\lambda)}_{g_n(\ell, s)}$$
where $\bar{f}=\ell f$ is given by
$$\bar{f}(s,1/\ell, \lambda)=\sum_{j=2}^\infty \frac{1}{j\ell^{j-1}} \left( \sum_{i=1}^\infty (-1)^i\frac{1}{\lambda^{i-1}}\frac{s_1^{2i}+s_2^{2i}}{(2i)!}\right)^j.$$

This expansion is valid as long as $|s| \leq \delta \sqrt{\lambda}$. Note that when $n=1$ we have
$$\bar{f}(s,1/\ell, \ell)=\sum_{j=2}^\infty \frac{1}{j\ell^{j-1}} \left( \sum_{i=1}^\infty (-1)^i\frac{1}{\ell^{i-1}}\frac{s_1^{2i}+s_2^{2i}}{(2i)!}\right)^j$$
and since $\ell > \lambda$ we have that 

$$\sum_{j=2}^\infty \frac{1}{j\ell^{j-1}} \left( \sum_{i=1}^\infty \frac{1}{\lambda^{i-1}}\frac{s_1^{2i}+s_2^{2i}}{(2i)!}\right)^j \leq \sum_{j=2}^\infty \frac{1}{j\lambda^{j-1}} \left( \sum_{i=1}^\infty \frac{1}{\lambda^{i-1}}\frac{s_1^{2i}+s_2^{2i}}{(2i)!}\right)^j.$$

Then 
$$|g_n(\ell, s)| \leq \lambda |h(s /\sqrt{\lambda})|+|\bar{f}(s, 1/\ell,\lambda)|\leq \lambda|h(s /\sqrt{\lambda})|+\frac{c|s|^4}{\lambda}$$
where $c$ can be chosen independent of $n$.
Note that 
$$\lambda\ h(s/ \sqrt{\lambda})=\lambda \sum_{i=2}^\infty (-1)^i\frac{1}{\lambda^i}\frac{s_1^{2i}+s_2^{2i}}{(2i)!}= \sum_{i=2}^\infty (-1)^i \frac{1}{\lambda^{i-1}}\frac{s_1^{2i}+s_2^{2i}}{(2i)!}$$
so $\lambda\ |h(s /\sqrt{\lambda})|=o(|s|^2)$ and so we can find $0< \epsilon \leq \delta$ such that for $|s|\leq \epsilon \sqrt{\lambda}$
$$ |g_n(s,\ell)| \leq \frac{|s|^2}{4}.$$

Let  $F_{\ell,n}(s) = e^{g_n(\ell,s)}-1$ and let 
$$\bar{p}_\ell(0)=\frac{1}{(2\pi)^2\lambda}\int_{\R^2} e^{-\frac{|s|^2}{2}} ds = \frac{1}{2\pi \lambda}$$
be the integral of a two-variable standard normal distribution (see for example Equation 2.2 in \cite{Lawler}). 

Then
\begin{eqnarray*}
P(v^\ell_{k_1}=v_{k_2}^\ell=0)&=&\frac{1}{(2\pi)^2} \iint_{[-\pi,\pi]^2} \phi(\theta)^\ell d\theta \\
&=&\frac{1}{(2\pi)^2} \iint_{[-\pi,\pi]^2}  \left(  1  -\frac{|\theta|^2}{2n} + \frac{1}{n} h(\theta)
   \right)^\ell d\theta\\
   &=& \frac{1}{(2\pi)^2\lambda} \iint_{[-\pi\sqrt{\lambda},\pi\sqrt{\lambda}]^2} e^{-|s|^2/2} (F_{\ell,n}(s) + 1) ds\\
   &=&    \frac{1}{(2\pi)^2\lambda}\left( A_n(\epsilon,\ell)+\iint_{|s|\leq \epsilon\sqrt{\lambda}} e^{-|s|^2/2} (F_{\ell,n}(s) + 1) ds \right)\\
&=&   \bar{p}_\ell(0) + B_n(\epsilon, \ell) + \frac{1}{(2\pi)^2\lambda}\left( A_n(\epsilon,\ell) + \iint_{|s| \leq \epsilon \sqrt{\lambda}} e^{-\frac{|s|^2}{2}} F_{\ell,n}(s) ds\right)
\end{eqnarray*}

where  
$$|A_n(\epsilon, \ell)|=\left|\iint_{[-\pi\sqrt{\lambda},\pi\sqrt{\lambda}]^2\setminus \{ |s|\leq \epsilon \sqrt{\lambda}\}}\phi(s/\sqrt{\lambda})^\ell ds\right| \leq C\lambda e^{-\beta\lambda}$$
where $C$ and $\beta$ do not depend on $n$
since $|\phi(\theta)| \leq 1- \frac{b}{n} |\theta|^2 \leq  e^{- \frac{b}{n} |\theta|^2}$ (where $b$ does not depend on $n$)  for all $\theta \in [-\pi, \pi]^2$
and so for $|s|\geq \epsilon \sqrt{\lambda}$ we have $\phi(s/\sqrt{\lambda}) \leq e^{-\beta/n}$.

Likewise
$$ |B_n(\epsilon, \ell)|= \left|\frac{1}{(2\pi)^2\lambda}\iint_{|s|> \epsilon \sqrt{\lambda}} e^{-|s|^2/2}ds\right|\leq C'e^{-\beta' \lambda}$$
where $\beta'$ and $C'$ do not depend on $n$.

Finally as long as $|s|\leq \lambda^{\frac{1}{8}}$ we have   $$|F_{\ell,n}(s)|\leq |e^{g_n(\ell, s)} -1| \leq C''g_n(\ell, s) \leq \frac{C'' |s|^4}{\lambda}$$ where $C''$ does not depend on $n$. Therefore  we have  
$$\left|\iint_{|s| \leq \lambda^{1/8}} e^{-\frac{|s|^2}{2}} F_{\ell,n}(s) ds\right| \leq \frac{C''}{\lambda}\int_{\R^2} |s|^4e^{-\frac{|s|^2}{2}} ds\leq  \frac{C'''}{\lambda}.$$

This leaves us only to estimate the integral for  $\lambda^{1/8} \leq |s| \leq \epsilon \sqrt{\lambda}$ where we have the bound $|F_{\ell,n}(s)| \leq e^{-\frac{|s|^2}{4}} +1$. The integral then can be estimated as follows
$$\left| \iint_{ \lambda^{1/8} \leq |s| \leq \epsilon \sqrt{\lambda}} e^{-\frac{|s|^2}{2}} F_{\ell,n}(s) ds\right| \leq 2 \iint_{|s| \geq  \lambda^{1/8}} e^{-\frac{|s|^2}{4}} ds \leq  \bar C e^{-\zeta  \lambda^{1/4}}.$$ 
This gives the desired result.
\end{proof}

\begin{cor}\label{cor:two_matching}
Suppose $\ell \in \omega(n)$. For fixed $k_1 \ne k_2$, $$\Prob(v^\ell_{k_1}=v^\ell_{k_2}=w^\ell_{k_1}=w^\ell_{k_2}=0) \sim K^4 \left( \frac{n}{\ell} \right)^2$$.
\end{cor}
\begin{proof}
This follows from Lemma \ref{lemma:two_zero} and the fact that 
$$\Prob(v^\ell_{k_1} = v^\ell_{k_2} = w^\ell_{k_1} = w^\ell_{k_2} = 0)=\Prob(v^\ell_{k_1}=v^\ell_{k_2}=0)\Prob(w^\ell_{k_1}=w^\ell_{k_2}=0).$$
\end{proof}

\begin{lemma}\label{lemma:hyperplane} Suppose $\ell \in \omega(n)$ and suppose $a_i=a_i(\ell)$  for $1 \le i \le n-1$, with $\Prob(a_{1} \ne 0) \to 1$ as $\ell \to \infty$. Then  $\Prob(a_1 v_1 + a_2 v_2 + \dotsb + a_{n-1} v_{n-1} = 0) \to 0$ as $\ell \to \infty$.
\end{lemma}
\begin{proof}
\begin{align*}
\Prob(\sum_{i=1}^na_iv_i = 0) &= \Prob(v_1 = - {\sum_{i=2}^n\frac{a_i}{a_1}v_i = 0}\mid a_1\neq 0)\Prob(a_1\neq0)\\ &\quad\quad\quad\quad\quad\quad\quad\quad + \Prob(\sum_{i=1}^na_iv_i = 0\mid a_1= 0) \Prob(a_1=0) \\
&\leq \Prob(v_1 = - {\sum_{i=2}^n\frac{a_i}{a_1}v_i = 0}\mid a_1\neq 0) +  \Prob(a_1=0) \\
&\le \Prob(v_1 = 0) + \Prob(a_1=0) \\
\end{align*}
{since the most likely value for $v_1$ is $0$ and therefore by Lemma \ref{lemma:prob_zero} this limit goes to zero.}
\end{proof}

\section{Asymptotic Abelianess}\label{sec:comm}
  In this section we prove Theorem \ref{theorem:abelian}. To check that $G_{\ell, n}$ is abelian we only need to check that $V,W$ commute.
Most of our analysis involves the notion of \emph{supercommuting}
that we defined in the introduction. Recall that for two words $V = V_1 V_2 \dotsm V_{\ell}$ and 
$W = W_1 W_2 \dotsm W_{\ell}$ with $V_i,W_i\in S$ to supercommute, every $V_i$ must commute with every $W_j$.

Clearly supercommuting is a sufficient (but not necessary) condition for commuting. However, when $\ell \in o(n)$, the probability of $V$ and $W$ commuting but not supercommuting goes to zero as $n \to \infty$. Therefore, when $\ell$ is in this class, these two notions of commuting are asymptotically equivalent. 

To prove this fact, we begin by defining the function
\[\sigma_i(Z) := \begin{cases}
1 & \text{if $Z = A_i$} \\
-1 & \text{if $Z = A_i^{-1}$} \\
0 & \text{otherwise.}
\end{cases}\]
Since multiplication in $U_n$ is additive on the superdiagonal elements,
\[v_i = \sum_{j = 1}^\ell \sigma_i(V_j) \quad w_i = \sum_{j = 1}^\ell \sigma_i(W_j).\]

In other words, the $i^{th}$ superdiagonal entry of $V$ is a count of the number of times one of $A_i^{\pm1}$ appears in the word $V=V_1\dots V_n$, where $A_i$ contributes $+1$, and its inverse $-1$. Since $\ell$ is growing more slowly than the size of our matrix (and hence more slowly than the size of our generating set $S$), 
the probability of seeing a particular $A_i$ in an $\ell$-step walk approaches zero. We make this precise in the following lemma. 
\begin{lemma}\label{lemma:sparse}
Suppose $\ell \in o(n)$. For fixed $1 \le i \le n-1$ and $Z= Z_1 Z_2 \dotsm Z_\ell$, where $Z_i \in S=\{A_1^{\pm1},\dots A_{n-1}^{\pm1}\}$, 
$$\Prob(\text{$\sigma_i(Z_j) \ne 0$ for some $1 \le j \le \ell$}) \to 0$$
 as $n \to \infty$.
\end{lemma}
\begin{proof}
For fixed $j$, \[\Prob(\sigma_i(Z_j) = 0) =\left( 1 - \frac{2}{2(n-1)}\right) = \left(1-\frac{1}{n-1}\right).\] Since the $Z_j$'s are independent, \[\Prob(\text{$\sigma_i(Z_j) = 0$ for all $j$}) = \left(1- \frac{1}{n-1}\right)^\ell.\] Since $\ell \in o(n)$, the limit of this probability is $1$, and so its negation---the probability that $\sigma_i(Z_j) \ne 0$ for some $j$---goes to $0$.
\end{proof}

Now suppose that $A_i$ appears at least once in our word $Z_1 Z_2 \dotsm Z_\ell$. Lemma \ref{lemma:sparse} implies that it, or its inverse, almost surely does not appear again.

\begin{cor}\label{cor:just-once}
Suppose $\ell \in o(n)$ and $Z = Z_1 Z_2 \dotsm Z_\ell$. For a fixed $1 \le i \le n-1$,  the $i^{\textrm{th}}$ superdiagonal entry $z_i$  of $Z$ satisfies  \[\Prob(z_i = \pm 1 \mid \text{$\sigma_i(Z_j) \ne 0$ for some $j$}) \to 1\] as $n \to \infty$.
\end{cor}
\begin{proof}
This follows from the fact that  $\Prob(z_i = \pm 1 \mid \sigma_i(Z_j) \ne 0\textrm{ for some }j \textrm{, and } 
 \sigma_i(Z_k)=0 \textrm{ for all } k \neq j))=1$
and $$\Prob( \sigma_i(Z_k)=0 \textrm{ for all } k \neq j)=  \left(1 - \frac{1}{n-1}\right)^{\ell-1} \to 1$$ as $\ell \to \infty.$
\end{proof}

\begin{lemma}\label{lemma:supercommute}
When $\ell \in  o(n)$, $$\Prob(\text{$V$ and $W$ commute but do not supercommute}) \to 0.$$
\end{lemma}
\begin{proof}
Note that
\begin{align*}
&\Prob(\text{$V$ and $W$ commute but do not supercommute}) \\
\le &\Prob(\text{$V$ and $W$ commute} \mid \text{$V$ and $W$ do not supercommute}).
\end{align*}
We will call this latter (conditional) event $\mc{C}$ and show that $\Prob(\mc{C}) \to 0$.

Let $A_i$ and $A_{i+1}$ be called \emph{neighboring} elementary matrices. If $V$ and $W$ do not supercommute, then Corollary \ref{cor:elementary_commute} implies the words $V$ and $W$ must contain neighboring matrices. Without loss of generality, this implies there must be some $1 < k \le n-1$ and some $1 \le i,j \le \ell$ such that $\sigma_{k-1}(W_i) \ne 0$ and $\sigma_{k}(V_j) \ne 0$. We bound $\Prob(\mathcal{C})$ above by considering the events $w_{k-1} \ne \pm 1$, $v_k \ne \pm1$ and the joint event $w_{k-1} = \pm 1, v_{k} = \pm 1$. While these three events are not mutually exclusive, they do cover all possibilities.
\begin{align*}
\Prob(\mathcal{C}) &\le \Prob(\mathcal{C} \mid w_{k-1} \ne \pm 1) \Prob(w_{k-1} \ne \pm 1) 
+ \Prob(\mathcal{C} \mid v_{k} \ne \pm 1) \Prob(v_{k} \ne \pm 1) \\
&+ \Prob(\mathcal{C} \mid w_{k-1}, v_{k} = \pm 1) \Prob(w_{k-1}, v_{k} = \pm 1).
\end{align*}
By Corollary \ref{cor:just-once} the first two terms go to $0$ and the last term goes to just $\Prob(\mathcal{C} \mid w_{k-1}, v_{k} = \pm 1)$. By Lemma \ref{lemma:commute}, this is at most
\begin{align*}
\Prob(\mathcal{C} \mid w_{k-1}, v_{k} = \pm 1) &\le \Prob\left(
w_{k-1}v_{k} -w_{k}v_{k-1}=0
\mid w_{k-1}, v_{k} = \pm 1\right) \\
&\le \Prob(w_{k} v_{k-1} \ne 0) \\
&\le \Prob(v_{k-1} \ne 0)
\end{align*}
and $\Prob(v_{k-1} \ne 0) \to 0$ by Lemma \ref{lemma:sparse}.
\end{proof}

\subsection{Part 1 of Theorem \ref{theorem:abelian}: when $\ell(n) \in o(\sqrt{n})$.}\label{osqrtn:subsec}
In this case,  we can use a counting argument to show that $V$ and $W$ supercommute. 

\begin{lemma} Assume that $\ell \in o(\sqrt{n})$,  $V = V_1 V_2 \dotsm V_{\ell},$ and $W = W_1 W_2 \dotsm W_{\ell}.$
Let $F$ be the number of pairs $i, j$ for which $V_i$ and $W_j$ fail to commute. Then the expected value $\Ex(F) \to 0$ as $n \to \infty$.
\end{lemma}
\begin{proof}
 Let $\gamma_{i,j}$ be an indicator random variable whose value is $1$ precisely when $V_i W_j \ne W_j V_i$.  By Corollary \ref{cor:elementary_commute}, for each $k$, there are at most $2$ values of $i$ such that $V_i$ does not commute with $A_k^{\pm 1}$. Since $W_j = A_k^{\pm 1}$ for some $1 \le k \le n-1$,  when $2 \leq k \leq n-2$, the probability that $V_i$ does not commute with $W_j$ is  $\frac{4}{2(n-1)} = \frac{2}{n-1}$; when $k$ is equal to $1$ or $n-1$, the probability is $\frac{2}{2(n-1)}=\frac{1}{n-1}$. Therefore the probability $P(V_i W_j \ne W_j V_i) \leq \frac{2}{n-1}$ for all $i$ and $j$.
Since  $F$ counts the number of non-commuting pairs $V_i, W_j$, we have
\begin{align*}
F &= \sum_{i = 1}^\ell \sum_{j = 1}^\ell \gamma_{i,j}. \\
\intertext{By linearity of expected value,}
\Ex(F) &= \sum_{i = 1}^\ell \sum_{j = 1}^\ell \Ex(\gamma_{i,j}). \\
&= \sum_{i = 1}^\ell \sum_{j = 1}^\ell \Prob(V_i W_j \ne W_j V_i) \\
&\leq  \sum_{i = 1}^\ell \sum_{j = 1}^\ell \frac{2}{n-1} \\
&\leq \ell^2 \left( \frac{2}{n-1} \right).
\end{align*}
Since $\ell \in o(\sqrt{n})$ then $\ell^2 \in  o(n)$ and
\[\lim_{n \to \infty} \Ex(F) \leq \lim_{n \to \infty} \frac{2 \ell^2}{n-1} = 0. \qedhere\]
\end{proof}

\begin{cor}\label{cor:ab-o-sqrt}
If $\ell \in o(\sqrt{n})$ then $V$ and $W$ supercommute asymptotically almost surely.
\end{cor}
\begin{proof} The elements
$V$ and $W$ supercommute precisely when every $V_i$ commutes with every  $W_j$,  that is when $F = 0$.
Since $F$ is a nonnegative integer random variable, and $\Ex(F) \to 0$ we have that $\Prob(F = 0) \to 1$.
\end{proof}

\subsection{Part 2 of Theorem \ref{theorem:abelian}: when $\ell = c\sqrt{n}$.}\label{csqrtn:subsec}
We start with a heuristic argument. For $V$ and $W$ to supercommute, $V_i$ must commute with $W_j$ for all $1 \le i,j \le \ell$. The probability that a given $V_i$ and $W_j$ commute is $1 - 2/(n-1)$ for most cases. Since there are $\ell^2$ such pairs, the probability that they all commute is
\[\left(1 - \frac{2}{n-1}\right)^{\ell^2} = \left(1 - \frac{2}{n-1}\right)^{c^2 n} \to \frac{1}{e^{2c^2}}.\]
This argument assumes independence of each $V_i, W_j$ pair commuting, which does not in general hold. However, we are able to show that limiting probability for abeilianess is nonetheless $1/e^{2c^2}$, as predicted.

If we fix the $V_i$'s, there is a specific set of $k$'s for which $A_k^{\pm 1}$ fails to commute with at least one $V_i$. Let $B$ be the number of such $k$'s; then since the $W_j$'s are chosen independently, the probability that all of them commute with $V$ is given by
\begin{equation}\label{eqn:supcomeqn} \Prob( V \textrm{ and } W\,  \textrm{supercommute} )=\left(1 - \frac{B}{n-1}\right)^\ell.\end{equation}

Now we have to say something about the distribution of $B$. Imagine a row of $n-1$ bins.  For each $V_i = A_k^{\pm 1}$, we put a ball in bin $k-1$ and a ball in bin $k+1$.
Then $B$ is the number of non-empty bins. Since there are $2\ell$ balls, two\footnote{When $V_i = A_1^{\pm 1}$ or $A_{n-1}^{\pm 1}$ only one ball is added; but this almost never happens as $n \to \infty$.} for each $V_i$, we have $2 \le B \le 2\ell$.  Let $D$ be the difference $2 \ell - B$. We will show that the expected value of $D$ approaches a constant.

\begin{lemma}\label{lemma:d-exp}
If $\ell= c \sqrt{n}$ then  $\Ex(D) \to 2c^2$ as $n \to \infty$.
\end{lemma}
\begin{proof}
Let $V=V_1V_2 \dotsm V_\ell$. First, we count the number $X$ of ``empty bins".
We write  $X = \sum X_i$, where 

\[ X_i = \begin{cases}
0 & \text{if $A_{i+1}$ or $A_{i-1}$ appears in the word V} \\
1 & \text{otherwise}.
\end{cases}\]

Note that the behaviors for the end bins (when $i=1$ or $i=n-1$) are slightly different than the other bins but asymptotically this difference will not be important.

Since each element $V_i$ is chosen independently, we have,

\[\Ex(X_i) =  \Prob(X_i=1)= \left(1-\frac{2}{n-1}\right)^\ell.\]

Therefore, $\Ex(X) = (n-1) \left(1-\frac{2}{n-1}\right)^\ell$. Since $B$ is the number of nonempty bins, $B+X = n-1$, and we have,

\[\Ex(B) = (n-1)-(n-1) \left(1-\frac{2}{n-1}\right)^\ell = (n-1)\left(1-\left(1-\frac{2}{n-1}\right)^\ell \right).\]

Finally, since $D$ is the difference $2 \ell - B$, the expected value of $D$ is $\Ex(D) = 2\ell - E(B)$. Taking the limit as $n$ goes to infinity gives the result. 
\end{proof}

In order to evaluate the limit of Equation \eqref{eqn:supcomeqn} as $\ell \to \infty$ we need to control  the size of $B=2\ell -D$. For this we consider two cases: when $D\geq \log{\ell}$ and when $D \leq \log{\ell}$.

\begin{lemma}\label{cor:d-bound}
If $\ell = c \sqrt{n}$ then $\Prob(D \ge \log \ell) \to 0$ as $n \to \infty$.
\end{lemma}
\begin{proof}
Markov's inequality tells us that $\Prob(D \ge \log \ell) \le \Ex(D) / \log \ell$. Since $\Ex(D)$ converges to a constant by Lemma \ref{lemma:d-exp} but $\log \ell$ grows without bound, this probability goes to $0$.
\end{proof}

\begin{lemma}\label{lemma:ab-sqrt-small-d}
If $\ell = c\sqrt{n}$ then $\Prob(\text{$G_{\ell, n}$ is abelian} \mid D < \log \ell) \to 1/e^{2c^2}$ as $n \to \infty$.
\end{lemma}
\begin{proof}
Recall (by Lemma \ref{lemma:supercommute} and Equation \ref{eqn:supcomeqn}) that
 \begin{eqnarray*}\lim_{n \to \infty}\Prob(\text{$G_{\ell, n}$ is abelian}) &=& \lim_{n \to \infty}\Prob(\text{$W$ and $V$ supercommute})\\
  &=& \lim_{n \to \infty} \left(1 - \frac{B}{n-1}\right)^\ell
 \end{eqnarray*}
and that by definition of $D$, $B = 2\ell - D$. Since  $0< D < \log \ell$, we have
\[\left(1 - \frac{2\ell }{n-1}\right)^\ell  \leq \left(1 - \frac{2\ell - D}{n-1}\right)^\ell  \leq \left(1 - \frac{2\ell - \log \ell }{n-1}\right)^\ell. \]
Using standard techniques (taking the logarithm and using L'H\^opital's rule) one can show that as $n \to \infty$ both the extreme functions limit to $1/e^{2c^2}$, and the result follows.

\end{proof}

\begin{lemma}\label{lemma:ab-sqrt}
If $\ell = c \sqrt{n}$ then $\Prob(\text{$G_{\ell, n}$ is abelian}) \to 1/e^{2 c^2}$ as $n \to \infty$.
\end{lemma}
\begin{proof} We have
\begin{align*}
\lim_{n \to \infty} \Prob(\text{$G_{\ell, n}$ is abelian}) &= \lim_{n \to \infty} \Prob(\text{$G_{\ell, n}$ is abelian} \mid D < \log \ell) \Prob(D < \log \ell) \\
&\qquad+ \lim_{n \to \infty} \Prob(\text{$G_{\ell, n}$ is abelian} \mid D \ge \log \ell) \Prob(D \ge \log \ell). \\
\intertext{By Lemma \ref{cor:d-bound} the second term goes to zero and the second factor of the first term goes to one, leaving just}
&= \lim_{n \to \infty} \Prob(\text{$G_{\ell, n}$ is abelian} \mid D < \log \ell) \\
&= \frac{1}{e^{2c^2}} 
\end{align*}
by Lemma \ref{lemma:ab-sqrt-small-d}.
\end{proof}

\subsection{Part 3 of Theorem \ref{theorem:abelian}: when $\ell \in \omega(\sqrt{n})$ and $\ell \in o(n)$.} 
By Lemma  \ref{lemma:supercommute} we know that when $\ell \in o(n)$ supercommuting is asymptotically the same as commuting.
Therefore to show that  asymptotically $G_{\ell, n}$ is almost never abelian we only need to show that $V$ and $W$ almost never supercommute. To show this, we consider $n-1$ ``bins", one for each $A_i$. We think of each element $V_i$ as a ball of a particular type, say red. Similarly each of the elements $W_i$ correspond to a blue ball. We throw the $\ell$ red balls, and $\ell$ blue balls into the $n-1$ bins, and look for a particular collision that implies $V$ and $W$ don't supercommute. To prove this, we will use the following Lemma which is a generalized (to 2 colors)  version of the probabilistic pigeonhole principle. A statement for $q$-colors appears in \cite{density-one-half}.

\begin{fact}[Lemma 5 in \cite{density-one-half}]\label{fact:pigeonhole}
Let $\mu$ be any probability measure on a set of size $n$. Let $z_1, \dotsc, z_{2\ell}$ be chosen randomly and independently using $\mu$. Then
\[\Prob(\exists \, i,j \,\, \textrm{with}\,\, i \le \ell < j, z_i = z_j) \ge 1 - 2 e^{-c\ell / \sqrt{n}}\]
for some universal constant c.
\end{fact}
In particular, when $\ell \in \omega(\sqrt{n})$, this probability approaches $1$ as $n \to \infty$.

\begin{lemma}\label{lemma:ab-large}
When $\ell \in \omega(\sqrt{n})$ as $n \to \infty$ the probability that $V,W$ supercommute goes to zero. 
\end{lemma}
\begin{proof}
Let $f$ be the function that takes $A_k^{\pm 1}$ to $k$, and define $2 \ell$ random variables $\set{z_i}$ as follows: when $i \le \ell$,
\begin{align*}
z_i &= f(V_i) \\
\intertext{and when $i > \ell$,}
z_i &= \begin{cases}
n-1 &\text{if $f(W_{i - \ell}) = 1$} \\
f(W_{i-\ell}) - 1 &\text{otherwise}
\end{cases}
\end{align*}
Then the conditions of Fact \ref{fact:pigeonhole} apply to the $z_i$'s, and so asymptotically almost surely there exist an $i$ and $j$ so that $i \le \ell < j$ and $z_i = z_j$. This means that either $z_i=f(V_i) = f(W_{j - \ell}) - 1=z_j$ or $f(V_i) = n-1$ and $f(W_{j - \ell}) = 1$. The latter case has probability $1/(n-1)$, and so as $n \to \infty$ we are almost surely in the former case. Thus $V_i = A_k^{\pm 1}$ and $W_{j - \ell} = A_{k+1}^{\pm 1}$. Then $V_i$ and $W_j$ do not commute, and so $V$ and $W$ do not supercommute. 
\end{proof}

\begin{cor} If $\ell = \omega(\sqrt{n})$ and $\ell = o(n)$ then $G_{\ell, n}$ is asymptotically almost surely nonabelian.
\end{cor}
\begin{proof}
By Lemma \ref{lemma:ab-large} the probability that $V,W$ supercommute goes to zero and therefore by Lemma \ref{lemma:supercommute}, $G_{\ell, n}$ is asymptotically almost surely nonabelian.
\end{proof}

\subsection{Part 3 of Theorem \ref{theorem:abelian}: when $\ell \in \omega(n)$}
In this case we need results from Section \ref{sec:dist} on the distribution of superdiagonal entries.
\begin{lemma}
When $\ell \in \omega(n)$ then $G_{\ell,n}$ is a.a.s. not abelian.
\end{lemma}
\begin{proof}
By Lemma \ref{lemma:commute}, if $v_1 w_2 \ne v_2 w_1$ then $G_{\ell,n}$ is not abelian. By Lemma \ref{lemma:prob_zero}, $\Prob(w_2 = 0) \sim K \sqrt{n/\ell} \to 0$. Then by Lemma \ref{lemma:hyperplane}, $\Prob(v_1 w_2 = v_2 w_1) = \Prob(v_1 w_2 - v_2 w_1 = 0) \to 0$, and so a.a.s. $v_1 w_2 \ne v_2 w_1$.
\end{proof}

\subsection{Part 3 of Theorem \ref{theorem:abelian}: when $k \leq \ell/n \leq M $}

To complete the proof of Theorem \ref{theorem:abelian} part 3, we need to consider functions $\ell$ which lie in  the complement of $o(n)$, and $\omega(n)$; we therefore consider  functions $\ell$ such that for large enough $n$, there exists constants $k$ and $M$ so that \[k \leq \frac{\ell}{n} \leq M.\] 
  
To show that $G_{\ell,n}$ is not abelian, it is sufficient to find $1 \le i \le n-2$ for which the condition of Lemma \ref{lemma:commute} fails; that is, there exists an $i$ so that $v_i w_{i+1} \neq v_{i+1} w_i$. To do this, we count a subset of pairs of  words $V$ and $W$ which have this property, and show that these pairs occur with high probability. 

\begin{lemma}\label{lemma:comm_config}
Suppose there exist constants $k$ and $M$ so that for large enough $n$, $k\leq \ell/n \leq M$. Then a.a.s. there is some $1 \le i \le n-2$ for which $v_i = \pm 1$, $v_{i+1} = 0$, $w_i = \pm 1$, and $w_{i+1} = \pm 1$.
\end{lemma}
\begin{proof}

We will look specifically for cases in which $V_j = A_i^{\pm 1}$ for precisely one $j$, $V_j \ne A_{i+1}^{\pm 1}$ for all $j$, $W_j = A_i^{\pm 1}$ for precisely one $j$, and $W_j = A_{i+1}^{\pm 1}$ for precisely one $j$. Note that words $V$ and $W$ of this form have $v_i = \pm 1$, $v_{i+1} = 0$, $w_i = \pm 1$, and $w_{i+1} = \pm 1$. Hence, by Lemma \ref{lemma:commute}, $V$ and $W$ will not commute. It'll be useful to have a  name for this sort of failure to commute, so we'll say this particular sort of pair $(V,W)$ has a  ``type $i$" configuration. Out strategy for this proof is to define a random variable $X$ which counts the expected number of type $i$ configurations for a pair of words $(V,W)$. We then show $E[X^2]/E[X]^2 \to 1$. It will be sufficient to consider only odd values of $i$, and as this makes some of the counting arguments simpler, we make this assumption.

Fix $i$. Let $S_i$ be the set of words $V$ of length $\ell$ which have  $V_j = A_i^{\pm 1}$ for precisely one $j$ and $V_j \ne A_{i+1}^{\pm 1}$ for all $j$. There are $\ell$ indices to choose for the location of $A_{i}^{\pm 1}$, two choices for the exponent on $A_i$, and after subtracting out the elements $A_i^{\pm 1} $ and $A_{i+1}^{\pm 1}$, we have $2(n-3)$ remaining generators to choose from for the remaining $\ell -1$ elements in the word $V$. Since the total number of words of length $\ell$ is $(2(n-1))^\ell$, we have
\[ P(S_i) = \frac{\ell(n-3)^{\ell-1}}{(n-1)^\ell} = \frac{\ell}{n-1} \frac{(n-3)^{\ell -1}}{(n-1)^{\ell-1}} =  \frac{\ell}{n-1} \left(1 - \frac{2}{n-1}\right)^{\ell-1} . \]

%

Let $T_i$ be the set of words $W$  for which $W_j = A_i^{\pm 1}$ and exactly one $j'$ for which $W_{j'} = A_{i+1}^{\pm 1}$. Then we have 
\[P(T_i) =  \frac{\ell(\ell-1)(n-3)^{\ell-2}}{(n-1)^\ell} = \frac{\ell (\ell -1)}{(n-1)^2} \left(1 - \frac{2}{n-1}\right)^{\ell-2} . \]

Since $V$ and $W$ are chosen independently, we have

\begin{equation}\label{eqn:bounds} 
P(S_i , T_i) =  \frac{\ell^2 (\ell -1)}{(n-1)^3} \left(1 - \frac{2}{n-1}\right)^{2\ell-3} .
\end{equation}

Now we compute the probability of $S_i \cap S_{i^\prime}$, for distinct $i$ and $i^\prime$. Counting words of this sort is where we use the convenience of only considering odd indices, so that $|i-i^\prime| \geq 2$.

\[ P(S_i \cap S_{i^\prime}) = \frac{\ell(\ell-1)(n-5)^{\ell-2}}{(n-1)^\ell} = \frac{\ell(\ell -1)}{(n-1)^2}  \left(1 - \frac{4}{n-1}\right)^{\ell-2} . \]

Similarly, we compute the probability of $T_i \cap T_{i^\prime}$. 

\[ P(T_i \cap T_{i^\prime}) = \frac{\ell(\ell-1)(\ell -2)(\ell-3)(n-5)^{\ell-4}}{(n-1)^\ell} = \frac{\ell(\ell -1)(\ell-2)(\ell -3)}{(n-1)^4}  \left(1 - \frac{4}{n-1}\right)^{\ell-4}. \]

Let $n' $ be the number of odd integers in $[1, n-2]$, and let $X$ be the number of odd values of $i$ for which a type $i$ configuration occurs in the pair $(V,W)$.
Define the random variable $X_i$ 
\[ X_i = \begin{cases}
1 &\text{if V is in}\,\, S_i \,\, \text{and} \,\,  \text{W in} \,\,T_i \\
0 &\text{otherwise}
\end{cases} \]
Then $X = \sum_{i = 1}^{n'} X_{2i-1}$ and  
\[ \Ex(X) = n' P(S_i,T_i). \]

Note that when $\ell$ is in the complement of $o(n)$, we have $\Ex(X) \to \infty$ as $n \to \infty$. (Also,  when $\ell$ is in $\omega(n)$, the expected value $E(x) \to 0$ as $n\to \infty$, hence this proof is not valid when $\ell$ is in this range.)

When $i \neq i'$, $X_iX_{i^\prime} = 1 $ if and only if $V $ is in $S_i \cap S_{i^\prime}$ and $W$ is in $T_i \cap T_{i^\prime}$. Therefore, 

\[ \Ex(X^2) = n' P(S_i,T_i)+ n'(n'-1) P(S_i \cap S_{i^\prime},T_i \cap T_{i^\prime}). \]
We now argue that $\Ex(X^2)/\Ex(X)^2 \to 1$ as $n \to \infty$. 
\begin{align*}
\frac{\Ex(X^2)}{\Ex(X)^2}  &= \frac{n' P(S_i,T_i)+ n'(n'-1) P(S_i \cap S_{i^\prime},T_i \cap T_{i^\prime})}{(n')^2P(S_i,T_i)^2}\\
&= \frac{1}{n^\prime P(S_i,T_i)} + \left( \frac{n^\prime - 1}{n^\prime}\right) \frac{P(S_i \cap S_{i^\prime},T_i \cap T_{i^\prime})}{P(S_i,T_i)^2}.
\end{align*}

When $\ell$ is bounded above by $Mn$, the first term goes to zero as $n \to \infty$. After simplifying a bit, we have, 

\begin{align*}
\frac{P(S_i \cap S_{i^\prime},T_i \cap T_{i^\prime})}{P(S_i,T_i)^2} &= \frac{(\ell -2)(\ell -3)}{\ell^2}  \left(\frac{n-5}{n-3}\right)^{2\ell -6} \left(\frac{n-3}{n-1}\right)^{-2\ell}
\end{align*}

When $\ell = cn$, the product of the later two functions limits to $1$. When can therefore conclude that $E[X^2]/E[X]^2 \to 1$ whenever $\ell$ is (eventually) bounded below by $kn$ and above by $Mn$. Since $E[X] \to \infty$, asymptotically almost surely  $X > 0$, meaning that there is some odd $i$ for which a type $i$  configuration occurs. 
 
\end{proof}

\begin{cor}
Suppose there exits constants $k$ and $M$ so that for large enough $n$, $k\leq \ell/n \leq M$; then a.a.s. $G_{\ell,n}$ is not abelian.
\end{cor}

\section{Full step}\label{sec:full}

To analyze whether our group $G_{\ell, n}$ has full step we rely heavily on the results from Section \ref{sec:dist}.

Define two families of indicator random variables $\delta$ and $\gamma$ as follows:
\[ \delta_{v,i} = \begin{cases}
1 & \text{if $v_i = 0$} \\
0 & \text{if $v_i \ne 0$}
\end{cases}
\quad
\delta_{w,i} = \begin{cases}
1 & \text{if $w_i = 0$} \\
0 & \text{if $w_i \ne 0$}
\end{cases}
\quad
\gamma_i = \begin{cases}
1 & \text{if $v_i = w_i = 0$} \\
0 & \text{if $v_i \ne 0$ or $w_i \ne 0$}
\end{cases}
\]

Note that $\gamma_i = \delta_{v,i} \delta_{w,i}$.


\subsection{Part 1 of Theorem \ref{theorem:full_step}: when $\ell \in o(n^2)$}

In this case we show that $G_{\ell,n}$ is a.a.s never full step but we separate the proofs into two subcases.  In Corollary \ref{cor:On} we consider the case when $\ell \in O(n)$ while in Lemma \ref{lem:notOn} we consider the case when $\ell \in \omega(n)\cap o(n^2)$. The following lemma is standard but is the basis for Corollary \ref{cor:On} so we include the proof. 

\begin{lemma}\label{lemma:cee_n_empty}
If $cn$ balls are thrown uniformly and independently into $n$ bins, there is a.a.s. at least one empty bin.
\end{lemma}
\begin{proof}
Let $X$ be the number of empty bins. Then $X = \sum_i X_i$ where
\[ X_i = \begin{cases}
1 &\text{if bin $i$ is empty} \\
0 &\text{otherwise.}
\end{cases} \]
Then
\begin{align*}
\Ex(X) &= n \Ex(X_i) \\
&= n \Prob(\text{Bin $i$ is empty}) \\
&= n \left( 1 - \frac{1}{n} \right)^{c n}
\end{align*}
and
\begin{align*}
\Ex(X^2) &= \Ex(X) + 2 \sum_{i \ne j} \Ex(X_i X_j) \\
&= n \left( 1 - \frac{1}{n} \right)^{c n} + 2\frac{n(n-1)}{2} \Prob(\text{Bins $i$ and $j$ are both empty}) \\
&= n \left( 1 - \frac{1}{n} \right)^{c n} + n(n-1) \left( 1 - \frac{2}{n} \right)^{c n}.
\end{align*}

Thus $\Ex(X) \to \infty$ and
\begin{align*}
\frac{\Ex(X^2)}{\Ex(X)^2} &= \frac{n \left( 1 - \frac{1}{n} \right)^{c n} + n(n-1) \left( 1 - \frac{2}{n} \right)^{c n}}{n^2 \left( 1 - \frac{1}{n} \right)^{2c n}} \\
&\sim \frac{ \left( 1 - \frac{2}{n} \right)^{cn} }{ \left( 1 - \frac{1}{n} \right)^{2cn} } \\
&\to 1,
\end{align*}
And so $\Prob(X = 0) \to 0$. Thus a.a.s. $X > 0$, and so there is at least one empty bin.
\end{proof}

\begin{cor}\label{cor:On}
If $\ell \in O(n)$, a.a.s. $G_{\ell,n}$ is not full-step.
\end{cor}
\begin{proof}
Let $V = V_1 \dotsm V_\ell$ and $W = W_1 \dotsm W_\ell$. Set up $n-1$ bins and put a ball in bin $i$ whenever some $V_j = A_i^{\pm 1}$ or $W_j = A_i^{\pm 1}$. Note that this process effectively throws in $2\ell$ balls uniformly and independently into the $n-1$ bins. Since $\ell \in O(n)$, there is some $c > 0$ for which $2\ell < c(n-1)$ for large enough $n$, and thus by Lemma \ref{lemma:cee_n_empty} there is an empty bin. This empty bin corresponds to some $i$ for which $v_i = w_i = 0$, and so by Lemma \ref{lemma:not_full} $G_{\ell, n}$ is not full-step.
\end{proof}

\begin{lemma}\label{lem:notOn}
If $\ell \in o(n^2)$ and $\ell \in \omega(n)$, a.a.s.\ $G_{\ell, n}$ is not full-step.
\end{lemma}
\begin{proof}
Let $X$ be the number of positions on the superdiagonal for which $V$ and $W$ both have a $0$. That is
\begin{align*}
X &= \sum_i \gamma_i \\
\Ex(X) &= \sum_i \Ex(\gamma_i) \\
&= n \Prob(v_i = w_i = 0). \\
\intertext{By Corollary \ref{cor:prob_matching},}
&\sim n K^2 \frac{n}{\ell} \\
& \sim K^2 \frac{n^2}{\ell} \\
&\to \infty
\end{align*}
when $\ell \in o(n^2)$. 

Also,
\begin{align*}
\Ex(X^2) &= \Ex\left[ \left( \sum_i \gamma_i \right)^2 \right] \\
&= \sum_i \Ex(\gamma_i) + 2 \sum_{i \ne j} \Ex(\gamma_i \gamma_j) \\
&= \sum_i \Prob(v_i = w_i = 0) + 2 \sum_{i \ne j} \Prob(v_i = v_j = w_i = w_j = 0). \\
\intertext{By Corollaries \ref{cor:prob_matching} and \ref{cor:two_matching},}
&\sim n K^2 \frac{n}{\ell} + n^2 K^4 \frac{n^2}{\ell^2} \\
&= K^2 \frac{n^2}{\ell} + K^4 \frac{n^4}{\ell^2}.
\end{align*}

Then
\begin{align*}
\frac{\Ex(X^2)}{\Ex(X)^2} &\sim \frac{K^2 \frac{n^2}{\ell} + K^4 \frac{n^4}{\ell^2}}{K^4 \frac{n^4}{\ell^2}}
\intertext{and since $\ell \in o(n^2)$ the second term dominates in the numerator to give us}
&\sim \frac{K^4 \frac{n^4}{\ell^2}}{K^4 \frac{n^4}{\ell^2}} \\
&\sim 1.
\end{align*}

Since $\Ex(X) \to \infty$ and $\Ex(X^2)/\Ex(X)^2 \to 1$ then $P(X > 0) \to 1$. So there is at least one $i$ for which $\gamma_i = 1$, that is $v_i = w_i = 0$. Then by Lemma \ref{lemma:not_full} we have that $G_{\ell, n}$ is not full-step.
\end{proof}

\subsection{Part 2 of Theorem \ref{theorem:full_step}: when $\ell \in \omega(n^3)$}

\begin{lemma}
If $\ell \in \omega(n^3)$, a.a.s.\ $G_{\ell, n}$ is full-step.
\end{lemma}
\begin{proof}
Let $X$ be the number of zeroes on the superdiagonal of $W$. That is
\begin{align*}
X &= \sum_i \delta_{w,i}. \\
\intertext{Then}
\Ex(X) &= \sum_i \Ex(\delta_{w,i}). \\
\intertext{Since the $\delta$ are identically distributed,}
&= n \Prob(w_i = 0). \\
\intertext{By Lemma \ref{lemma:prob_zero},}
&\sim n K \sqrt{\frac{n}{\ell}} \\
&\sim K \sqrt{\frac{n^3}{\ell}} \\
&\to 0
\end{align*}
when $\ell \in \omega(n^3)$. This means that $\Prob(X = 0) \to 1$, and so a.a.s.\ none of the $w_i$ are $0$.

Now, for $G_{\ell, n}$ to be full-step (that is, step $n-1$), the $(n-2)$-commutator subgroup must have a nontrivial element. In particular, consider the commutator
\[ C^{n-2}=\underbrace{[W, [W, \dotsc [W}_\text{$n-2$},V]]]. \]
As we saw in Example \ref{ex:commutator} in Section \ref{sec:prelim}
the upper-right corner entry of $C^{n-2}$ is given by 
\[ c^{n-2}_{n,n} = K_1 v_1 w_2 w_3 \dotsm w_{n-1} + K_2 w_1 v_2 w_3 \dotsm w_{n-1} + \dotsb + K_{n-1} w_1 \dotsm w_{n-2} v_{n-1}\]
where each $K_i = \binom{n-1}{i}$ with alternating signs. Since the $w_i$ and $K_i$ are a.a.s. nonzero and $\ell \in \omega(n)$, Lemma \ref{lemma:hyperplane} says that $\Prob(c^{n-2}_{n,n} = 0) \to 0$ and thus a.a.s. $c^{n-2}_{n,n} \ne 0$, making $C^{n-2}$ nontrivial.
\end{proof}

\bibliographystyle{amsplain}
\bibliography{matrixmodel}

\providecommand{\bysame}{\leavevmode\hbox to3em{\hrulefill}\thinspace}
\providecommand{\MR}{\relax\ifhmode\unskip\space\fi MR }
\providecommand{\MRhref}[2]{%
  \href{http://www.ams.org/mathscinet-getitem?mr=#1}{#2}
}
\providecommand{\href}[2]{#2}
\begin{thebibliography}{1}

\bibitem{randnilp1}
Matthew Cordes, Moon Duchin, Yen Duong, and Andrew~P. S\'anchez, \emph{Random
  nilpotent groups 1}, To appear in International Mathematics Research Notices
  (2015).

\bibitem{tgroups}
Willem~A. de~Graaf and Werner Nickel, \emph{Constructing faithful
  representations of finitely-generated torsion-free nilpotent groups}, J.
  Symbolic Comput. \textbf{33} (2002), no.~1, 31--41. \MR{1876310
  (2003j:20009)}

\bibitem{density-one-half}
Moon Duchin, Kasia Jankiewicz, Shelby Kilmer, Samuel Leli\`evre, John~M Mackay,
  and Andrew~P. S\'anchez, \emph{A sharper threshold for random groups at
  density one-half}, To Appear in Groups, Geometry and Dynamics (2014).

\bibitem{Hall}
Philip Hall, \emph{The {E}dmonton notes on nilpotent groups}, Queen Mary
  College Mathematics Notes, Mathematics Department, Queen Mary College,
  London, 1969. \MR{0283083 (44 \#316)}

\bibitem{Lawler}
Gregory~F. Lawler and Vlada Limic, \emph{Random walk: a modern introduction},
  Cambridge Studies in Advanced Mathematics, vol. 123, Cambridge University
  Press, Cambridge, 2010. \MR{2677157 (2012a:60132)}

\bibitem{mks}
Wilhelm Magnus, Abraham Karrass, and Donald Solitar, \emph{Combinatorial group
  theory}, revised ed., Dover Publications, Inc., New York, 1976, Presentations
  of groups in terms of generators and relations. \MR{0422434 (54 \#10423)}

\bibitem{Ollivier}
Yann Ollivier, \emph{A {J}anuary 2005 invitation to random groups}, Ensaios
  Matem\'aticos [Mathematical Surveys], vol.~10, Sociedade Brasileira de
  Matem\'atica, Rio de Janeiro, 2005. \MR{2205306 (2007e:20088)}

\end{thebibliography}

\end{document}